\documentclass[preprint,numbers,12pt]{article}
\usepackage{newlfont}
\usepackage{a4wide}
\usepackage{amsmath}
\usepackage{inputenc}
\usepackage{amsfonts}
\usepackage{amssymb}
\usepackage{amsthm}
\usepackage{newlfont}
\usepackage{graphicx}
\usepackage{subfigure}
\usepackage{mathrsfs}
\usepackage[numbers]{natbib}
\usepackage{soul}

\usepackage{color}

\usepackage{comment}

\usepackage[running]{lineno}
\newtheorem{theorem}{Theorem}
\newtheorem{lemma}{Lemma}

\newtheorem{corollary}{Corollary}

\newtheorem{definition}{Definition\rm}
\newtheorem{remark}{Remark}

\theoremstyle{definition}

\newcommand{\PP}{{\mathbb P}}
\renewcommand{\AA}{{\mathcal A}}

\newcommand{\BB}{{\mathfrak B}}

\newcommand{\RR}{{\mathcal R}}
\newcommand{\IR}{{\mathrm{I\!R}}}

\newcommand{\IN}{{\mathrm{I\!N}}}

\newcommand{\LL}{{\mathcal L}}

\newcommand{\HH}{{\mathcal H}}

\newcommand{\TT}{{\mathcal T}}

\renewcommand{\SS}{{\mathcal S}}
\newcommand{\CC}{{\mathscr C}}
\newcommand{\CCE}{\sigma({\mathscr C}\cup \{E\})}
\newcommand{\supp}{\mathop{\mathrm{supp}}}
\newcommand{\co}{\mathop{\mathrm{co}}}
\newcommand{\cl}{\mathop{\mathrm{cl}}}
\newcommand{\ext}{\mathop{\mathrm{ext}}}

\newcommand{\IE}{\mathop{\mathrm{I\!E}}}

\newcommand{\CF}{{\mathcal C}}

\newcommand{\NWE}{{\mathcal N}}
\begin{document}


\title{A bang-bang principle for the conditional expectation vector measure}
\author{Youcef Askoura$^\dag$ And Mohammed Sbihi$^\ddag$}
\maketitle

{\small $^\dag$LEMMA, Universit\'e Paris II, 
4 rue Blaise Desgoffe, 75006 Paris. youcef.askoura@u-paris2.fr\\
$^\ddag$ENAC - Universit\'e de Toulouse, 7 avenue Edouard Belin, 31055 Toulouse.  mohammed.sbihi@enac.fr}

\begin{abstract}
We prove a conditional expectation bang-bang principle.  Based on properties of the conditional expectation vector measure, we establish that the conditional expectation of a set-valued mapping coincides with the conditional expectation of the set of selections of its extreme points part. As a by-product, we obtain straightforwardly a purification principle.  
\end{abstract}



\section{Introduction}
The classical Lyapunov convexity theorem is a fundamental result for both mathematical economics and the theory of games and optimal control: it is used to provide purification principles and the existence of bang-bang controls. 
It states that a nonatomic finite dimensional vector measure has a convex and compact range. Generally, when the range space of the vector measure is infinite dimensional, the classical convexity theorem of Lyapunov fails. However, an infinite dimensional version of such result established by \citet{KNO75} can be used to establish a purification principle as in \citet{ASK19}. 

Furthermore, many works succeeded to recover this property and to establish adequate bang-bang and purification results by using Maharam types and saturated measure spaces; see \citet{GrP13,KhS13,KhS14,KhS15,KhS16,SAG16} for infinite dimensional Lyapunov convexity theorems, purification processes and applications to the integration of set-valued mappings, equilibrium theory and control systems.  Refer to \citet{KeS09} and the literature therein for further prevalent studies on this direction.

\vspace{0,5cm}
\citet{SOL70}, \citet{DyE77} and  \citet{HeS18b} proved that  the conditional expectation vector measure satistfy a  Lyapunov convexity theorem   by generalizing appropriately the ``nonatomicity" notion. More precisely, some condition is assumed on a $\sigma$-algebra (to be atomless) over its sub-$\sigma$-algebra. The classical nonatomicity becomes atomless over the trivial $\sigma$-algebra. This notion is reformulated as ``setwise coarser sub-$\sigma$-algebra" in \citet{HeS14} where it is successfully used to establish existence of pure strategy equilibria of incomplete information games and some purification process. 
It  is renamed  as ``nowhere equivalent" $\sigma$-algebras and used in \citet{HeS18} to handle basic properties such as  convexity, closure, and compactness for the distribution of correspondences that are necessary for applied work. He and Sun \cite{HeS18} show their condition to be necessary and sufficient for the existence of pure strategy equilibria in some formulations of large games.  

\vspace{0,3cm}
Working with this generalized nonatomicity, and the Lyapunov convexity property, we establish a bang-bang  principle for  the conditional expectation of a set-valued mapping and prove that it generates straightforwardly a purification principle.  Such results are very useful in economics as showed by \citet{HeS18b}.

\vspace{0,3cm}
The paper is organized as follows: in Section \ref{SCEVM},  we revisit and summarize some properties of the conditional expectation vector measure. In Section \ref{SBB}, we prove the bang-bang principle for the conditional expectation vector measure and obtain as a corollary a purification result. We end by an Appendix in Section \ref{SAP}, where we collect some results used in a relevant way and where we give an alternative elementary proof of the Lyapunov property of the conditional expectation vector measure based on Knowles' characterization of Lyapunov vector measures with infinite dimensional range spaces \cite{KNO75}.

\section{The conditional expectation vector measure}\label{SCEVM}
Consider a probability space $(\Omega,\AA,\mu)$, where $\AA$ is a $\sigma$-algebra. Denote by $\AA^+$ the set of elements $E\in \AA$ with $\mu(E)>0$. For a sub-$\sigma$-algebra $\CC\subset \AA$ and $E\in \AA$, denote by $\CC_E=\{A\cap E : A\in \CC\}$ the trace $\sigma$-algebra of $E$ in $\CC$. We denote by $\bar{E}$ the complement of the  set $E$. Given two sets $E_1,E_2\in \AA$, $E_1\Delta E_2$ refers to the symmetric difference between $E_1$ and $E_2$. In order to simplify notation, when $\CC$ is a sub-$\sigma$-algebra of $\AA$, we keep $\mu$ instead of $\mu_\CC$ for the restriction of $\mu$ to $\CC$.  Since different  $\sigma$-algebra are used simultaneously, the $L^p$-spaces are denoted in a more specific way. For instance $L^1(\Omega,\AA,\mu)$, or  $L^1(\Omega,\CC,\mu)$ to refer to the underlying $\sigma$-algebra. 
Given $f \in L^1(\Omega,\AA,\mu)$, we denote by $\IE^\mu(f|\CC)$, or simply  $\IE(f|\CC)$ if  no confusion may occur, the conditional expectation of $f$ relatively to $\CC$ under the measure $\mu$.

The set of real numbers is denoted by $\IR$ and that of non-negative integers by $\IN$. The set $\IN^*$ refers to $\IN\setminus \{0\}.$

The following notion is introduced in \cite{HeS18} as a ``nowhere equivalence" between $\sigma$-algebra. Even if ``nowhere equivalence" reflects intuitively the introduced notion, we rename it in order to reflect the influence of the underlying measure $\mu$.

\begin{definition}[\cite{HeS18}]\label{MUC}A sub-$\sigma$-algebra $\CC\subset \AA$ is said to be $\mu$-coarser than $\AA$ iff for every $E\in \AA^+$, there exists $E_0\in \AA$, $E_0\subset E$ such that $$\mu(F\Delta E_0)>0,\forall F\in \CC_E. $$
\end{definition}
It is worth mentioning that when $\CC_E$ is $\mu$-complete, the condition ``there exists $E_0\in \AA$, $E_0\subset E$ such that $\mu(F\Delta E_0)>0,\forall F\in \CC_E$" is equivalent to the following: there exists   $E_0\in \AA^+$ such that $E_0\subset E$ and $E_0\notin \CC_E$. So, in this case, Definition \ref{MUC} expresses the absence of $\CC$-atoms in $\AA$ in the sense of \citet{HaN66} (p. 443).

\vspace{0,5cm}
This notion generalizes the nonatomicity notion. We see straightforwardly that $\mu$ is nonatomic iff $\{\emptyset,\Omega\}$ is $\mu$-coarser than $\AA$. In case of nonatomicity of $\mu$, every finite sub-$\sigma$-algebra of $\AA$ is $\mu$-coarser than $\AA$. It is easy to construct an infinitely generated $\mu$-coarser sub-$\sigma$-algebra. For instance consider the unit interval $[0,\; 1]$ endowed with its Borel $\sigma$-algebra $\BB([0,\;1])$ and the Lebesgue measure $\lambda$. Obviously, $\sigma(\{]\frac{1}{n+1},\frac{1}{n}], n\in \IN^*\})$ is $\lambda$-coarser than $\BB([0,\;1])$.
It is clear as well that the existence of $\mu$-coarser sub-$\sigma$-algebras in $\AA$ implies necessarily that $\mu$ is nonatomic. Hence, the previous notion is a strengthening of the nonatomicity notion. While the simple nonatomicity of $\mu$ does not suffice to handle some situations, the $\mu$-``coarserness" notion can be used successfully. We will see below that this is the case for the Lyapunov convexity property for the conditional expectation vector measure.

An equivalent concept, as proved in  \cite{HeS18} (Lemma 1, p 616), is the following notion:

\begin{definition}[\cite{SOL70,DyE77,HoK84}]\label{ATO} $\AA$  is said to be atomless over a sub-$\sigma$-algebra $\CC\subset \AA$  iff for every $E\in \AA^+$, there exists $E_0\in \AA$, $E_0\subset E$ such that: on some set of $\AA^+$

$$0<\mu(E_0|\CC)<\mu(E|\CC), $$
where the conditional probability $\mu(D|\CC)$ corresponds to $\IE{}^\mu(\chi_D|\CC)$
and $\chi_D$ is the indicator function of the set $D$.
\end{definition}

\vspace{0,5cm}
Let $X$ be a Banach space. We comply with \cite{DiU77} and call a vector measure an additive function $G:\AA\rightarrow X$. A vector measure $G:\AA\rightarrow X$ is said to be $\mu$-continuous, and this is denoted by $G<\!\!<\mu$,  iff $\underset{\mu(E)\rightarrow 0}{\lim} G(E)=0$.

\begin{definition}
An X-valued countably additive vector measure $G$ is said to be Lyapunov if for every $E\in \AA$, $G(\AA_E)=\{G(A\cap E): A\in \AA\}$ is convex and weakly compact.
\end{definition}

It is known that (\cite{KLU73}, Theorem 2, p.49) the range of a Lyapunov vector measure can be described by: $$G(\AA_E)=\left\{\int_E fdG : f \text{ is measurable defined from }E \text{ to }[0,1]\right\}.$$
Conversely, every measure satisfying this property is Lyapunov, since for every $E\in \AA^+$, the operator $g\mapsto \int_E g dG$ defined from $L^\infty (E,\AA_E,\mu)$ into $X$ is continuous for the weak* topology on $L^\infty(E,\AA_E,\mu)$ and the weak topology on $X$ \cite{DiU77}. 

Let $\CC\subset \AA$ be a sub-$\sigma$-algebra. It is clear that the vector measure $G_\CC:\AA \rightarrow L^1(\Omega,\CC,\mu)$ defined by $$E\mapsto \IE(\chi_E|\CC)$$
is bounded, countably additive and since $\IE(\chi_A|\CC)=0$ for every $\mu$-null set $A$, $G_\CC<\!\!< \mu$ (\cite{DiU77}, Theorem 1, p.10). Furthermore, we can check easily that for every $g\in L^\infty(\Omega,\AA,\mu)$, $\int_E g dG_\CC= \IE(g\chi_E|\CC)$. This results follows straightforwardly from the norm-continuity of  $\IE(\cdot|\CC):L^\infty(\Omega,\AA,\mu)\rightarrow L^1(\Omega,\CC,\mu)$ and the denseness of the set of simple functions in $L^\infty(\Omega,\AA,\mu)$.

Let $f\in L^1(\Omega,\AA,\mu)$ and define the vector measure $F_\CC : \AA\rightarrow L^1(\Omega,\CC,\mu)$ by $F_\CC(E)=\IE(f\chi_E|\CC)$. Similarly, $F_\CC$ is bounded, countably additive, $F_\CC<\!\!<\mu$, and for every $g\in L^\infty(\Omega,\AA,\mu)$, $\int_E g dF_\CC= \IE(gf\chi_E|\CC)$.  By an abuse of notation we designate below $F_\CC$ by $fG_\CC$ and then perform the identification $F_\CC=fG_\CC$.

Let us now consider the vector case $f=(f_1,f_2,\dots,f_n )\in {L^1 (\Omega,\AA,\mu)^n}$ and consider the vector measure  $fG_\CC$ defined by 

$$E\mapsto (f_1G_\CC(E), f_2 G_\CC(E), \dots, f_n G_\CC(E)).$$

The following theorem is obtained in \cite{SOL70}. A proof of an  equivalent statement is given in  \cite{HeS18b}. For the sake of completeness, we give a new and elementary proof in the Appendix.

\begin{theorem}[\cite{SOL70}]\label{theorem-induction}
Let $\NWE$ be the set  of sub-$\sigma$-algebras that are $\mu$-coarser than $\AA$.
For all $n\in \IN^*$, the following statement holds true
$$P(n): \forall f=(f_1,f_2,\dots,f_n)\in {L^1(\Omega,\AA,\mu)}^n, \forall \CC\in \NWE, fG_\CC \mbox{ is Lyapunov}.$$
\end{theorem}

In connection with some results in the literature, let us reformulate the previous theorem in a slightly more general way. Its proof, in the Appendix, shows that it is a simple consequence of Theorem \ref{theorem-induction}. 

\begin{theorem}\label{THLG}
Let $\mu=(\mu_1,\dots,\mu_n)$ be a vector of finite measures on $(\Omega,\AA)$ and $\CC$ a sub-$\sigma$-algebra of $\AA$. Assume that $\CC$ is $\mu_i$-coarser than $\AA$, for every $i=1,...,n$. Let $(f_1,\dots,f_n)\in \prod_{i=1}^n L^1(\Omega,\AA,\mu_i)$. Consider 
the vector measure  $F_\CC:\AA\rightarrow \prod_{i=1}^n L^1(\Omega,\CC,\mu_i)$ defined by 
$$
F_\CC(E)= (\IE{}^{\mu_1}(f_1\chi_E | \CC),\IE{}^{\mu_2}(f_2\chi_E | \CC), \dots, \IE{}^{\mu_n}(f_n\chi_E | \CC)),
$$

Then, $F_\CC$ is Lyapunov. 
\end{theorem}

\begin{remark}
Observe that the property of being $\mu$-coarser is in some sense hereditary. Specifically, for every $E\in \AA^+$, if $\CC$ is $\mu$-coarser than $\AA$, then $\CC_E$ is $\mu$-coarser than $\AA_E$ too. Hence, if the range of $fG_\CC$ is convex and weakly compact, then the range of $fG_\CC$ restricted to $\AA_E$, for $E\in \AA^+$, is necessarily convex and weakly compact. That is, Theorem \ref{theorem-induction} is not different from Soler's (\cite{SOL70}, proposition IV.12), stating simply that the range of $fG_\CC$ is convex and weakly compact. Note that Dynkin and Evstigneev (\cite{DyE77}, Lemma 5.1) showed the scalar version of such a result.

In the same spirit, we can see analogously that by setting $f_i=1$ for every $i$ in Theorem  \ref{THLG}, we recover Proposition 1 in \cite{HeS18b}.\footnote{Consider without loss of generality that Proposition 1 in \cite{HeS18b} corresponds exactly to the special case of Theorem  \ref{THLG}, obtained by setting $f_i=1$, for every $i$.} But the converse implication is also true. Indeed, it is observed in \cite{HeS18b}, that Proposition 1 in \cite{HeS18b} implies Theorem \ref{theorem-induction}, when the functions $f_i$ are positive.
The general case can be handled by decomposing each $f_i$ into its positive and negative  parts $f_i=f_i^+ - f_i^-$   and considering the vector measure $F'$ defined by 
$$F'(E)=(\IE(f_1^-\chi_E|\CC),\dots,\IE(f_n^-\chi_E|\CC), \IE(f_1^+\chi_E|\CC),\dots,\IE(f_n^+\chi_E|\CC) )$$ which is Lyapunov as stated above. Hence by Theorem \ref{TH4} in the Appendix, for every $E\in \AA^+$ there exists $g\in L^\infty(\Omega,\AA,\mu)$ vanishing outside of $E$, such that $\int_{E} g dF'=0$. Then, $\IE(f_i^{-}\chi_{E}g|\CC)=\IE(f_i^{+}\chi_{E}g|\CC)=0$ for every $i$, which is not else but $\int_{E} g dfG_\CC=0$. Consequently, by Theorem \ref{TH4} 	again, the vector measure $fG_\CC$ is Lyapunov. 

At this stage, Proposition 1 in \cite{HeS18b} implies Theorem \ref{theorem-induction}. However,  Theorem \ref{theorem-induction} is a particular case of Theorem \ref{THLG} which is itself a simple consequence of the former as will be shown in the Appendix.

\end{remark}

\section{Bang-bang principle for conditional expectations}\label{SBB}

By default, topological spaces are endowed with their Borel $\sigma$-algebra denoted for a topological space $E$ by $\mathcal{B}(E).$ The convex hull  operator is denoted by $\co(\cdot)$ and the topological closure by $\cl(\cdot)$.

\vspace{0,3cm}

Let $(\Omega,\AA,\mu)$ be a probability space, $X$ a separable Banach space and $T:\Omega\rightarrow X$ a set-valued map.Then, $T$ is said to be \emph{measurable} iff $T^-{O}:=\{\omega\in \Omega : T(\omega)\cap O\neq \emptyset\}\in\AA$ for all open set $O\subset X$.  A \emph{measurable selection} of $T$ is a measurable function $f : \Omega\rightarrow X$, such that $f(\omega)\in T(\omega)$, for all $\omega\in \Omega$.  
If $T:\Omega\rightarrow X$ is  with convex values, the extreme points of values of $T$ are represented in the following set-valued map $\ext(T):\Omega\rightarrow X$ by $\ext(T)(\omega)=\{y\in T(\omega) : y \text{ is an extreme point of }T (\omega)\}$.
We say that $T$ is $L^1$-bounded if there is an integrable function $\varphi : \Omega\rightarrow \IR$ such that $$\|x\|\leq \varphi(\omega), \forall \omega\in \Omega, \forall x\in T(\omega).$$
 
By assuming that $T$ is $L^1$-bounded, define the integral of $T$ by the set of Bochner integrals of its measurable selections
$$ \int_\Omega T d\mu =\left\{\int_\Omega f d\mu : f \text{ is a measurable selection of } T\right\}.$$

Note that since $X$ is separable, the measurability and strong measurability (or $\mu$-measurability) of $X$-valued functions\footnote{Being a $\mu$-a.e. limit of a sequence of simple functions.} are identical (\cite{DiU77}, Theorem 2, p. 42). We denote by $L^1(\mu,X)$ the set of equivalence classes of $\mu$-a.e. equal Bochner integrable functions defined from $\Omega$ to $X$. 
In the particular case of $X=\IR^n$, integration of $\IR^n$-valued function is understood coordinatewise or equivalently in the Bochner sense.  The specification of the norm $\|\cdot\|$ used in $\IR^n$ is not relevant for the results below.

\begin{definition}Consider a probability space $(\Omega,\AA,\mu)$ and a separable Banach space $X$. Let $T:\Omega\rightarrow X$ be a measurable and $L^1$-bounded set-valued map, with compact nonempty values and $\CC\subset \AA$  a sub-$\sigma$-algebra. Denote the set of  measurable  selections of $T$ by  $\SS_T$ and denote the set of their conditional expectations by 

$$\IE(\SS_{T}|\CC)=\{\IE(f|\CC): f\in \SS_T\}.$$

\end{definition}

The following theorem is a bang-bang principle for the conditional expectation operator. 

\begin{theorem}\label{BB}Consider a probability space $(\Omega,\AA,\mu)$, where $\AA$ is a $\mu$-complete $\sigma$-algebra. Let $T:\Omega\rightarrow \IR^n$ be a measurable, $L^1$-bounded set-valued map, with compact convex nonempty values and $\CC\subset \AA$ a sub-$\sigma$-algebra which is $\mu$-coarser than $\AA$. Then, 

$$\IE (\SS_T |\CC)=\IE(\SS_{\ext(T)}|\CC),  \mu\text{-a.e. on } \Omega.$$
\end{theorem}

\begin{proof}
It is obvious that $\IE(\SS_{\ext(T)}|\CC)\subset\IE (\SS_T |\CC), \mu\text{-a.e. on } \Omega$. Let us consider the converse inclusion. Let $h\in \SS_T$. From Theorem \ref{FINSEL} (in the Appendix), there is $h_i,\; \alpha_i,\; i\in I=\{1,...,n+1\}$ such that $h_i$ is a measurable selection of $\ext(T)$ for every $i\in I$ and the weight functions $\alpha_i : \Omega\rightarrow [0,\; 1],i\in I,$ are measurable and satisfy $\sum_{i\in I} \alpha_i=1$. 
Furthermore, $h=\sum_{i\in I}\alpha_ih_i$. Since $T$ is $L^1$-bounded, the functions $h_i,i\in I,$ are integrable. 

\vspace{0,3cm}
From now, we work with equivalence classes, and $\mu$-a.e. equal functions are confused.
We prove thereafter that there exists a finite measurable partition $B_i,i\in I,$ of $\Omega$ satisfying 
\begin{equation}\label{EQOMEXT}
\IE(h|\CC)=\underset{i\in I}{\sum}\IE(\alpha_i h_i|\CC)=\underset{i\in I}{\sum} \IE(\chi_{B_i}h_i|\CC), \mu\text{-a.e. on }\Omega. 
\end{equation}
Consider the vector measure $G$ defined on $\AA$ by $$G(E)=\IE[(h_1,...,h_{n+1})\chi_E| \CC].$$

From Theorem \ref{theorem-induction}, $G$ is Lyapunov. Hence, from Lemma \ref{PLY} (in the Appendix) there exists a measurable partition $B_1,....,B_{|I|}$ of $\Omega$ such that, for every $i\in I$, 

$$G(B_i)=\int_{\Omega} \alpha_i d G.$$
That is, for every $i,j\in I$, 

$$\IE(\chi_{B_i}h_j|\CC)=\IE(\alpha_i h_j|\CC), \mu\text{-a.e. on }\Omega.$$

By summing up the diagonal elements (over the same indices $i=j$), we obtain,

\begin{equation*}
\IE(h|\CC)=\underset{i\in I}{\sum}\IE(\alpha_i h_i|\CC)=\underset{i\in I}{\sum} \IE(\chi_{B_i}h_i|\CC), \mu\text{-a.e. on }\Omega.
\end{equation*}

This proves Equation \eqref{EQOMEXT}.  

Define on $\Omega$ the function $f$ by setting $f_{|B_i}=h_i$, for every $i\in I$. We obtain a measurable function $f :\Omega\rightarrow \IR^n$ satisfying

$$\IE(f|\CC)=\IE(h|\CC), \mu\text{-a.e. on }\Omega.$$
Furthermore, by construction, $f(\omega)\in \ext(T(\omega)),$ $\mu$-a.e. on $\Omega$. 
\end{proof}

Theorem 5 in \cite{HeS18b} establishes that if $M:\Omega \rightarrow \mathbb{R}^n$ is a set-valued map satisfying assumptions of Theorem \ref{BB} except the convexity of its values, then  $\IE (\SS_M |\CC)=\IE(\SS_{\co(M)}|\CC)$.  Since $T$ has convex compact values (in Theorem \ref{BB}),  $T=\co(\ext(T))$. So, we are tempted to write $\IE (\SS_T |\CC)=\IE(\SS_{\co(\ext(T))}|\CC)=\IE(\SS_{\ext(T)}|\CC)$. But this conclusion is not correct since the set of extreme points is not necessarily closed, and the set-valued map $\ext(T)$ does not satisfy necessarily all the requirements of Theorem 5 in \cite{HeS18b}. That is Theorem \ref{BB} cannot be deduced straightforwardly from this result.  On the contrary,   Theorem 5 in \cite{HeS18b} can be deduced from Theorem \ref{BB}.

\begin{corollary}[\cite{HeS18b}]
Apart from the convexity of values of $T$, suppose that the assumptions of Theorem 3 are satisfied. Then,
$$\IE (\SS_T |\CC)=\IE(\SS_{\co(T)}|\CC).$$
\end{corollary}
\begin{proof} The set-valued map $\omega\mapsto \co(T(\omega))$ is measurable (\citet{CaV77}, Theorem III.40), with convex compact values. It is obviously $L^1$-bounded as well. Hence, from Theorem \ref{BB}, $\IE(\SS_{\co(T)}|\CC)=\IE(\SS_{\ext(\co(T))}|\CC)$. 

But, since $T$ has compact values, we have $\ext(\co(T(\omega)))\subset T(\omega)$, for every $\omega\in \Omega$, and hence $\IE(\SS_{\ext(\co(T))}|\CC)\subset \IE(\SS_{T}|\CC)$. That is, $\IE(\SS_{\co(T)}|\CC)\subset \IE (\SS_T |\CC)$.
The converse inclusion being obvious, we deduce that $\IE (\SS_T |\CC)=\IE(\SS_{\co(T)}|\CC)$.
\end{proof}

\begin{remark}Consider a probability space $(\Omega,\AA,\mu)$, where $\AA$ is a $\mu$-complete $\sigma$-algebra, and a separable Banach space $X$. Let $T:\Omega\rightarrow X$ be a measurable set-valued map, with compact  
nonempty values and $\CC\subset \AA$ a sub-$\sigma$-algebra. 
\begin{itemize}
\item Denote and define the set of  $\CC$-measurable and integrable $\mu$-a.e. selections of $T$ by  

$$\SS^1_{T}(\CC)=\{f\in L^1(\mu,X): f \text{ is } \CC-\text{measurable and } f(\omega)\in T(\omega), \mu-\textit{a.e.}\}.$$
Denoted simply by $\SS^1_{T}$ when $\CC=\AA$.

\item The conditional expectation of  $\SS^1_{T}$ is defined by $$\IE (\SS^1_T |\CC)=\{\IE (f |\CC) : f\in \SS^1_T\}.$$
\end{itemize}

Following Hiai and Umegaki (\cite{HiU77}, Theorem 5.1), for a measurable $L^1$-bounded  set-valued map $T$, there is a unique set-valued map $\IE(T|\CC)$, call it  \emph{conditional expectation} of $T$ satisfying

$$\SS^1_{\IE(T|\CC)}(\CC)=\cl(\IE (\SS^1_T|\CC)),$$
where the closure is taken with respect to the norm in $L^1(\mu,X)$.

Since  $T$ is assumed to be $L^1$-bounded in Theorem \ref{BB}, the set of measurable selections $\SS_T$ of $T$ coincides with the set of measurable and integrable $\mu$-a.e. selections $\SS^1_T$ of $T$. This remains also true for $\SS^1_{\ext(T)}$.  
Therefore, the previous theorem establishes, under its assumptions, that  
\begin{equation}\label{EspConEq}
\SS^1_{\IE(T|\CC)}(\CC)=\cl(\IE (\SS^1_T |\CC))=\cl(\IE(\SS^1_{\ext(T)}|\CC)),   \mu\text{-a.e. on } \Omega.
\end{equation}

But, since $T$ has convex compact values, $\SS^1_{T}$ is clearly convex and norm closed. Then, it is weakly closed. Furthermore, the $L^1$-boundedness of $T$ implies the $L^1$-boundedness of $\SS^1_{T}$ then its uniform integrability\footnote{Recall that a subset $\HH\subset L^1(\mu,\IR^n)$,  is said to be uniformly integrable iff $\underset{M\rightarrow +\infty}{\lim}\int_{\|f\|>M} \|f\|d\mu=0,\text{ uniformly in }f\in \HH.$ 
Equivalently, $\HH$ is bounded in $ L^1(\mu,\IR^n)$ and $\underset{\mu(E)\rightarrow 0}{\lim}\int_E \|f\|d\mu=0$, uniformly in $f\in \HH$.}. Consequently,  $\SS^1_{T}$ is weakly compact which implies that the convex set $\IE(\SS^1_{T}|\CC)$ is weakly compact, so norm closed,  as $\IE(.|\CC)$ is continuous for the weak  topologies.  By noticing that  $\IE(\SS^1_T |\CC))=\IE(\SS^1_{\ext(T)}|\CC)$, Equation \eqref{EspConEq} becomes
$$\SS^1_{\IE(T|\CC)}(\CC)=\IE (\SS^1_T |\CC)=\IE(\SS^1_{\ext(T)}|\CC),  \mu\text{-a.e. on } \Omega.$$

\end{remark}

If $\mu$ is nonatomic, then $\CC=\{\emptyset,\Omega\}$ is $\mu$-coarser than $\AA$ as already mentioned in section \ref{SCEVM}. In this particular case, Theorem  \ref{BB} reduces to the following well-known result (\cite{KhS14}, Theorem 3.1). 
\begin{corollary}
Consider a probability space $(\Omega,\AA,\mu)$, where $\AA$ is a $\mu$-complete $\sigma$-algebra. Let $T:\Omega\rightarrow \IR^n$ be a measurable, $L^1$-bounded set-valued map, with compact convex nonempty values. Assume that $\mu$ is nonatomic. Then, 
$$\int_\Omega T d\mu=\int_\Omega \ext(T)d\mu.$$ 
\end{corollary}

Based on the bang-bang principle (Theorem \ref{BB}), we are going to deduce a purification principle analogous to that established by He and Sun  (\cite{HeS14}, Proposition 2 and\cite{HeS18b}, Theorem 7). Let $(\Omega,\AA,\mu)$ be a probability space, where $\AA$ is a $\mu$-complete $\sigma$-algebra, and $A$ be a Polish space. Denote by $\PP(A)$ the set of Borel probability measures on $A$. Consider the space of Young measures\footnote{More generally the space of measures on $\Omega\times A$ having marginals on $\Omega$ equal to $\mu$ is referred to as the space of Young measures. We define, here, the space of disintegrated Young measures. However, under our assumptions (metrizability of $A$), all measure on the product having its marginal on $\Omega$ equal to $\mu$ can be disintegrated, see Theorem 3.2, p. 179 in \citep{FlG12} and \citep{VAL73}. Hence, we can consider that we do not modify the definition of this space.}  $\RR(\Omega,A)$ with respect to $(\Omega,\AA,\mu)$ and $A$: the set of functions $\delta : \Omega\rightarrow \PP(A)$ such that $\omega\mapsto \delta_\omega(B)$ is measurable, for every $B\in \BB(A)$. Endow $\RR(\Omega,A)$ with the topology of stable convergence\footnote{Since $A$ is metrizable Suslin, we do not make distinction between the different topologies known on $\RR(\Omega,A)$ : the $S, M, N, W$ defined in \citep{CRV04}, because they are identical, see the Portmanteau Theorem in \citep{CRV04}.}, \emph{i.e.,} the coarsest topology making continuous the functions $\delta \mapsto \int_\Omega \left[\int_{A}\phi d \delta_\omega\right] d\mu$, where $\phi :\Omega\times A\rightarrow \IR$ is a Carath\'eodory integrand ($\phi(\cdot, a)$ is measurable for every $a\in A$ and $\phi(\omega,\cdot)$ is continuous on $A$ for every $\omega\in \Omega$) which is $L^1$-bounded,  \emph{i.e.}, there exists a $\mu$-integrable function $h :\Omega\rightarrow \IR_+$ such that $|\phi (\omega,x)|\leq h(\omega)$, for every $(\omega,x)\in \Omega\times A$. Following (\citet{VAL90}, Lemma A2), the measurability condition defining Young measures is equivalent to the measurability of the functions $\omega \mapsto \delta_\omega (\phi)=\int_{A} \phi d\delta_\omega$, $\phi\in \CF_b(A)$, where $ \CF_b(A)$ is the set of continuous and bounded functions defined on $A$. We associate to every measurable function $f:\Omega\rightarrow A$, a (degenerate) Dirac-Young measures $\delta^f$, defined at every $\omega\in \Omega$ by $\delta^f_\omega(B)=1$ if $f(\omega)\in B$ and $\delta^f_\omega(B)=0$ otherwise. Denote by $D(\Omega,A)$ the set of Dirac-Young measures.

Given $V:\Omega\times A\rightarrow \IR^n$ and $\delta\in \RR(\Omega, A)$ denote by $\delta.V$ the function defined on $\Omega$ by $$\delta.V(\omega)=\int_A V(\omega,a) d\delta_\omega(a),$$

provided that the integral is well defined. Denote $f.V$ the similarly defined function for a measurable function $f:\Omega\rightarrow A$. That is, $f.V(\omega)=V(\omega,f(\omega)).$

\begin{corollary}\label{THMP}Let $(\Omega,\AA,\mu)$ be a probability space, where $\AA$ is a $\mu$-complete $\sigma$-algebra and $A$ a compact metric space and $n\in \IN^*$. Consider  a function $V:\Omega\times A\rightarrow \IR^n$ such that $V(\cdot,a)$ is measurable, for every $a\in A$, and $V(\omega,\cdot)$ is continuous on $A$ for every $\omega\in \Omega$, that is $V=(V_1,...,V_n)$ is a vector consisting of $n$ Carath\'eodory integrands. Assume further that $V_i$ is $L^1$-bounded, for every $i$, and $\CC$ is $\mu$-coarser than $\AA$. 
Then, for every $\delta\in \RR(\Omega,A)$, there exists a measurable function $f:\Omega\rightarrow A$ such that 
\begin{itemize}
\item $\IE( \delta.V |\CC)=\IE(f.V|\CC)$, and
\item $f(\omega)\in \supp(\delta_\omega),$ for $\mu$-a.e. $\omega\in \Omega$. 
\end{itemize}
\end{corollary}

\begin{proof}In this proof, some arguments dealing mainly with measurability considerations are taken from \cite{ASK19}. 

\vspace{0,3cm}

Since $\delta_\omega$ has a compact support and $V(\omega,\cdot)$ is continuous, for every $\omega\in \Omega$~: 
\begin{equation}\label{IMCO}
\co(V(\omega,\supp(\delta_\omega)))\text{ is compact for every }\omega\in \Omega.
\end{equation}

Fix a function $\bar{\delta} : \Omega\rightarrow \PP(A)$ representing the $\mu$-a.e. equal equivalence class of $\delta$.

Observe now that the set-valued map $\omega\mapsto \supp(\bar\delta_\omega)$ is measurable. Indeed, for every open set $U$ of $A$, 
$\{\omega\in \Omega : \supp(\bar\delta_\omega)\cap U\neq\emptyset\}=\{\omega\in \Omega : \bar\delta_\omega(U)>0\}$.
Hence, by definition of $\RR(\Omega,A)$,  
$\omega\mapsto \supp(\bar\delta_\omega)$ is measurable.

\vspace{0,3cm}
Consider the function $\bar p:\Omega \rightarrow \IR^n$ defined for every $\omega\in \Omega$ by 

$$\bar{p}(\omega)=\int_A V(\omega,a) d\bar\delta_\omega(a).$$

Let us check that $\bar{p}$ is measurable. For every fixed $\omega\in \Omega, V(\omega,\cdot)$ is continuous from the compact metric space $A$ into $\IR^n$,  hence, $\int_A V(\omega,a) d\bar\delta_\omega(a)$ is well defined.  The function $\omega\mapsto V(\omega,\cdot)$ is $\AA$-$\BB(\CF(A))$ measurable (\citet{AlB06}, Theorem 4.55, p. 155), where $ \CF(A)$ is the set of continuous functions defined on $A$. 
Then, since $\bar\delta$ is $\AA$-weak* measurable, the function $\omega\mapsto \int_A V(\omega,a) d\bar\delta_\omega(a)=\bar{p}(\omega)$ is measurable as a composition of measurable functions : $\omega\mapsto (V(\omega,\cdot),\bar\delta_\omega)$ from $\Omega$ to $\CF(A)\times \PP(A)$ obviously measurable and ($\varphi, q)\mapsto \int_A \varphi d q$ from $\CF(A)\times \PP(A)$ to $\IR$ which is continuous for the product of the norm by the weak* topology, then measurable. Hence $\bar{p}$ is  measurable.   
From the mean value theorem for the integral of vector valued functions (for instance \citet{DiU77}, Corollary 8, p. 48) and using \eqref{IMCO}, 
\begin{equation}\label{CCH}
\bar{p}(\omega)\in \co(V(\omega,\supp(\bar{\delta}_\omega))),\text{ for }\text{ every }\omega\in \Omega.  
\end{equation}
Remark now that $\omega\mapsto V(\omega,\supp(\bar\delta_\omega))$ is measurable. Indeed, since $\omega\mapsto \supp(\bar\delta_\omega)$ is measurable with compact values, it possesses a Castaing representation. That is, for every $\omega\in \Omega$,  $\supp(\bar\delta_\omega)=\cl(\{\sigma_k(\omega):k\in \IN\})$, where $\sigma_n$ are measurable functions from $\Omega$ to $A$, see (\citet{CaV77}, Theorem III.7, p. 66). But, from the continuity of $V(\omega,\cdot)$, $\omega\in \Omega$, $V(\omega,\supp(\bar\delta_\omega))=\cl(V(\omega,\{\sigma_k(\omega):k\in \IN\}))$ for every $\omega\in \Omega$. Clearly, $\omega\mapsto V(\omega,\{\sigma_k(\omega):k\in \IN\})$ is measurable as a countable union of measurable functions. Hence, $\omega\mapsto V(\omega,\supp(\bar\delta_\omega))$ is measurable (\citet{CaV77}, Theorem III.4). Therefore, $\omega\mapsto \co (V(\omega,\supp(\bar\delta_\omega)))$ is measurable, (\citet{CaV77}, Theorem III.40).  

Since $V(\cdot,\supp(\bar\delta_\cdot))$ has compact values, it results from Theorem \ref{BB} that there is a measurable selection $h$ of $\ext(\co(V(\cdot,\supp(\bar\delta_\cdot)))) \subset V(\cdot,\supp(\bar\delta_\cdot))$  such that
$$\IE(\delta.V|\CC)=\IE( \bar{p}(\cdot)|\CC)=\IE(h|\CC), \mu\text{-a.e. on }\Omega.$$

An application of the implicit measurable selection theorem (\citet{CaV77}, Theorem III.38) allows to find a measurable selection\footnote{Theorem III.38 in \citet{CaV77} establishes $\hat{\AA}$ measurable selections, where $\hat{\AA}$ is the universal completion of $\AA$. But, since $\AA$ is assumed to be $\mu$-complete, these selections are necessarily $\AA$-measurable.} $f$ of $\omega\mapsto \supp(\bar\delta_\omega)$, defined on $\Omega$, such that $h(\omega)=V(\omega,f(\omega))$ for every $\omega\in \Omega$.  Therefore, 

$$\IE(\delta.V|\CC)=\IE(f.V|\CC), \mu\text{-a.e. on }\Omega.$$

Furthermore, by construction, $f(\omega)\in \supp(\delta_\omega),$ $\mu$-a.e. on $\Omega$. 

\end{proof}

\begin{remark}
Corollary \ref{THMP} is obtained by He and Sun (\cite{HeS18b}, Theorem 7). A different version is obtained in (\cite{HeS14}, Proposition 2, p. 134), where a $\CC$-measurability restriction for $\delta\in\RR(\Omega,A)$ is assumed but the space $A$ is a general Polish Space. It is obtained in a more general setting in \cite{JaN19}.
\end{remark}

An immediate density result can be inferred from the previous corollary.  Endow $\RR(\Omega,A)$ by the following topology $\TT$ described by: a net $(\delta_\lambda)_{\lambda\in \Lambda}$ converges to $\delta$ iff for every $L^1$-bounded Carath\'eodory integrand $\varphi$ on $\Omega\times A$,
\begin{itemize}
\item[($\TT$)] $\IE( \delta^\lambda.\varphi|\CC)$ converges in $L^1(\Omega,\CC,\mu)$ to $\IE( \delta.\varphi|\CC)$. 
\end{itemize}

\begin{corollary}\label{Cor-density} Assume that $A$ is a compact metric space and  $\CC$ is $\mu$-coarser than $\AA$, then $D(\Omega,A)$ is dense in $\RR(\Omega,A)$ for the topology $\TT$. 
\end{corollary}
The proof is omitted, since the result follows straightforwardly from Corollary \ref{THMP}.

An interesting particular case is that of $\CC=\{\emptyset,\Omega\}$, where  $\IE(f|\CC)=\int_\Omega f d\mu$,  the topology $\TT$ is not else but the Stable topology. 
Therefore, if $\mu$ is nonatomic, which also means that $\CC=\{\emptyset,\Omega\}$ is $\mu$-coarser than $\AA$, we recover by Corollary \ref{Cor-density}  the well-known denseness result of $D(\Omega,A)$ in $\RR(\Omega, A)$ with respect to the Stable topology \cite{YOU37}. Refer to  \cite{CRV04} (Theorem 2.2.3, p. 40) for a more general version.

\section{Appendix}\label{SAP}
We recall in this section some results used in a relevant way in this paper.  

\begin{theorem}\label{TH4} [\citet{KNO75} or see Theorem 4, p. 263 in \citep{DiU77}] 
Let $G:\AA\rightarrow X$ be a countably additive vector measure and $\mu$ a real positive finite measure on $\AA$ such that for all $E\in \AA$, 
\begin{equation}\label{EQT4}
\mu(E)=0\Leftrightarrow G(E\cap F)=0,\forall F\in \AA.
\end{equation} 
Then, the following assertions are equivalent~:
\begin{itemize}
\item[1)] For all $E\in \AA^+$, the operator $f \mapsto \int_E f dG$ from $L^\infty(E,\mu)$ into $X$ is not a one-to-one.
\item[2)] For all $A\in \AA,$ $G(\AA_{|A})=\{G(F\cap A) :F\in \AA\}$ is convex and weakly compact. 
\end{itemize}
\end{theorem}

The previous theorem works by assuming instead of \eqref{EQT4}, that $G<\!\!<\mu$, see for instance \cite{ASK19}, Remark 1.

\begin{theorem}[Corollary 7, p. 265 in \citep{DiU77}] \label{COR7}
A countably additive vector measure $G:\AA\rightarrow X$ is Lyapunov iff for every $E\in \AA$, there is $F\in \AA$ such that $G(E\cap F)=G(E)/2$.
\end{theorem}

\begin{lemma}[Lemma 2 in \cite{ASK19}]\label{PLY} Let $X$ a Banach space and $G :\AA\rightarrow X$ be a Lyapunov vector measure and $\alpha_i :\Omega\rightarrow [0,\; 1]$, $i=1,...,n,$  be measurable functions, such that $\overset{n}{\underset{i=1}{\sum}}\alpha_i(\omega)=1,$ $G$-a.e. on $\Omega$.\footnote{Except on a $G$-null set. Note that $A\in\AA$ is said to be $G$-null iff $G(A\cap F)=0$ for every $F\in \AA$.} Then, there exists a measurable partition : $E_i,i=1,...,n$, of $\Omega$ such that, for every $i=1,...,n$, $\int_\Omega \alpha_i(\omega) dG=G(E_i)$.
\end{lemma}

\begin{theorem}[Theorem IV.11, p. 101, in \citep{CaV77}] \label{FINSEL} Let $(\Omega,\AA,\mu)$ be a probability space, where $\AA$ is $\mu$-complete and $T :\Omega \rightarrow \IR^n$ a measurable set-valued map with nonempty convex compact values. Define the set-valued map $\ext(T):\Omega\rightarrow \IR^n$ by  $\ext(T)(\omega)=\{a \in \IR^n: a\text{ is an extreme point of } T(\omega)\}$.
Let $s:\Omega\rightarrow X$ be a measurable selection of $T$. 
Then, there exist $n+1$ measurable selections $h_i:\Omega\rightarrow \IR^n$, $i=1,...,n+1$, of $\ext(T)$ and $n+1$ weight functions $\alpha_i\in \LL^\infty(\Omega,\mu,[0,1])$, $i=1,...,n+1$,  such that, 
$$\overset{n+1}{\underset{i=1}{\sum}}\alpha_i (\omega)=1, \text{ and }s(\omega)=\overset{n+1}{\underset{i=1}{\sum}}\alpha_i (\omega)h_i(\omega),\text{ for  every }\omega\in \Omega,$$ 

where, $\LL^\infty(\Omega,\mu,[0,\;1])$ refers to the set of measurable and bounded functions defined from $\Omega$ to $[0,\;1]$.
\end{theorem}

\subsection{Proofs of Theorems \ref{theorem-induction} and \ref{THLG}}
We produce thereafter a further elementary proof of Theorem \ref{theorem-induction} and deduce from it Theorem \ref{THLG}. Let us start with some needed lemmas.

\begin{lemma}\label{CARSIG}
Let $\CC$ be a sub-$\sigma$-algebra of $\AA$ and $E\in \AA^+$. Then,
\begin{enumerate}
\item The $\sigma$-algebra  $\CCE$, generated by $\CC\cup\{E\},$ is given by
$$
\CC_E \oplus \CC_{\bar{E}}:=\{A\subset \Omega: A \cap E \in \CC_E \mbox{ and }
A\cap\bar{E}\in \CC_{\bar{E}} \}. 
$$
Hence, $B\in \sigma(\CC\cup \{E\})$ iff there is  $C_1,C_2\in \CC$ such that $B=(C_1\cap E)\cup(C_2\cap \bar E).$ 
\item $\CC$ is $\mu$-coarser than $\AA$ iff $\CCE$ is $\mu$-coarser than $\AA$. 
Consequently, $\CC$ is $\mu$-coarser than $\AA$ iff $\sigma(\CC\cup\{E_i : i \text{ in a finite set }I\})$ is $\mu$-coarser than $\AA$.

\end{enumerate}
\end{lemma}
\begin{proof}
1) If $A\subset\Omega$ is such that $A \cap E \in \CC_E$  and $A\cap\bar{E}\in \CC_{\bar{E}}$ that is $A \cap E =C\cap E$
  and  $A\cap\bar{E}= C'\cap \bar{E}$ for $C,C'\in \CC$, then $A=(C\cap E) \cup (C'\cap \bar{E})\in \CCE,$ which establishes the inclusion $\CC_E \oplus \CC_{\bar{E}}\subset \CCE$. As obviously $\CC_E \oplus \CC_{\bar{E}}$ is a  $\sigma$-algebra, it remains to prove $\CC \cup\{E\} \subset \CC_E \oplus \CC_{\bar{E}}.$  To do so, let $A\in\CC \cup\{E\}.$ If $A=E$, $ A\cap E= E\in \CC_E$ and $A\cap \bar{E}=\emptyset\in \CC_{\bar{E}}.$ 
If $A \in \CC$, $A=(A\cap E)\cup (A\cap\bar{E})\in \CC_E \oplus \CC_{\bar{E}}$.

\vspace{0,3cm} 
2) Suppose that $\CC$ is $\mu$-coarser than $\AA$. Let $D\in \AA^+$. Suppose that $\mu(D\cap E)>0$. Then, there exists an $\AA$-measurable subset $D_0$ of $D\cap E$ such that 
\begin{equation}\label{EQ1L1}
\mu(D_0\bigtriangleup (D\cap E\cap C))>0, \quad\forall C\in \CC.
\end{equation}  
 Let $B\in  \CCE$, i.e. $B=(C_1\cap E)\cup(C_2\cap\bar{E})$ for $C_1,C_2\in \CC.$ Thus, 
$D_0\bigtriangleup (D\cap B)= D_0\bigtriangleup [(D\cap E\cap C_1)\cup (D\cap \bar E\cap C_2)]$ and since $D_0\subset  D\cap E$,  $D_0\bigtriangleup (D\cap E\cap C_1)\subset [D_0\Delta (D\cap B)]$. 

Then, from \eqref{EQ1L1},
$$
\mu(D_0\bigtriangleup (D\cap B))>0,\quad \forall B\in \CCE.
$$
The remaining case $\mu(E\cap D)=0$  can be handled in the same manner as previously by reasoning on $\bar{E}\cap D$ which is non-negligible. This means that $\sigma(\CC\cup\{E\})$ is $\mu$-coarser than $\AA$.
The converse is trivial, since $\CC\subset \sigma(\CC\cup\{E\})$.

\end{proof}

\begin{lemma}\label{JLY}
Let $\CC\subset \AA$ be a sub-$\sigma$-algebra that is $\mu$-coarser than $\AA$. For every $f\in  {L^1(\Omega,\AA,\mu)}$, the vector measure  defined from $\AA$ to $L^1(\Omega,\CC,\mu)$ by 
$$
fG_\CC(E)=\IE(f\chi_E|\CC)
$$
is well defined, countably additive and Lyapunov. 
\end{lemma}

\begin{proof} As noticed in section \ref{SCEVM}, $fG_\CC$ is bounded, countably additive, and for every $g\in L^\infty(\Omega,\AA,\mu)$, $\int_E g dfG_\CC= \IE(fg\chi_E|\CC).$  

According to Theorem \ref{TH4}, we have to prove that given an $E\in \AA^+$, there exists a function  $\bar g\in L^\infty(\Omega,\AA)$ vanishing outside of $E$, with $\|\bar g\|_\infty>0$ and such that $\IE(\bar g f\chi_E|\CC)=0$. Let $E\in \AA^+$. If $f=0, \mu$-a.e. on $E$, the result is obvious. Otherwise, there is $E'\subset E$, $E'\in \AA^+$ and $\varepsilon>0$, such that $|f|>\varepsilon$ on $E'$. Without loss of generality, we assume that $E'=E$ below, since we can perform the analysis below on $E'$ instead of $E$.

The sub-$\sigma$-algebra $\CC$ being $\mu$-coarser than $\AA$, there exists $E_0\in \AA$, $E_0\subset E$ such that $\mu(E_0\bigtriangleup F)>0$ for all $F\in \CC_E.$ 

Consider the function $g$ defined by $g= \chi_{E_0}-\IE(\chi_{E_0}|\CCE)$. 
Observe that $\IE(g\chi_E|\CC)=\IE[\IE(g\chi_E|\CCE|\CC)]=0$.
From $0\leq \chi_{E_0}\leq \chi_E$, it follows that $0\leq \IE(\chi_{E_0}|\CCE)\leq \IE(\chi_{E}|\CCE)=\chi_E$ for $\mu$-almost every $\omega\in \Omega$,  which implies that $g$ vanishes outside of $E$. 
The set  $F$ defined by $F= E\cap \{\omega\in \Omega: \IE(\chi_{E_0}|\CCE)(\omega)=1\}$ is a subset of $E$ and is $\CCE$-measurable. Hence, by Lemma \ref{CARSIG}, $F\in \CC_E$.
Consequently $\mu(E_0\bigtriangleup F)>0.$ Since $g\neq 0$ on $E_0\bigtriangleup F$, it follows that $\|g\|_\infty>0$. 

Set $\bar g=\frac{g}{f}$. Then, $\bar g\in L^\infty(\Omega,\AA,\mu)$, $\|\bar g\|_\infty>0,$ $\bar g$ vanishes outside of $E$ and $\IE(f\bar g\chi_E|\CC)=\IE(g\chi_E|\CC)=0$.
This ends the proof. 
\end{proof}

\begin{proof}[Proof of Theorem \ref{theorem-induction}] To shorten notation, we write $G$ (respectively $G_E$, $fG$, $fG_E$) instead of $G_\CC$ (respectively $G_{\CCE}$, $fG_\CC$, $fG_{\CCE})$.

We give a proof by induction on $n$. The statement $P(1)$ is nothing else but Lemma \ref{JLY}. Assume  that $P(n)$ is true. Let $f=(f_1,f_2,\dots,f_{n+1})\in {L^1(\Omega,\AA,\mu)^{n+1}}$. Let $E\in \AA^+$. According to Theorem \ref{TH4}, we have to construct a function $g\in L^\infty(\Omega,\AA,\mu)$ vanishing outside of $E$, with $\|g\|_\infty>0,$  such that $\IE(fg\chi_E|\CCE)=\IE(fg|\CCE)=0.$ This will end the proof since $\IE(fg|\CC)=\IE(\IE(fg|\CCE)|\CC)=0$.
Without loss of generality, we can assume that $\sum_{i=1}^{n+1} |f_i|>0$ for $\mu$-almost every $\omega\in E$:
  otherwise, there is $E'\in\AA^+$, $E'\subset E$, such that $f_i\equiv 0$ on $E'$, for every $i=1,...,n+1$. Set in this case $g=\chi_{E'}$. We can assume as well that every $f_i,i=1,...,n+1,$ has a constant sign on $E$:
  \begin{equation}\label{eq-signFi}
f_i= s_i |f_i| \mbox{ with } s_i\in\{-1,0,1\} \mbox{ for all } i=1,\cdots,n+1,
\end{equation}
  since we can always find such an $\AA$-measurable subset $E'$ of $E$ with a strictly positive measure and perform all the forthcoming analysis on $E'$ instead of $E$.  
Let $h_i= \frac{|f_i|}{\sum_{j=1}^{n+1} |f_j|}\chi_E$, which clearly belongs to ${L^1(\Omega,\AA,\mu).}$  

If  $h_i =  \IE(h_i|\CCE), \mu-a.e.$, that is $h_i$ is $\CCE$-measurable up-to a $\mu$-null set, for all $i =1,\dots, n+1,$ then it results that $\IE(h_i \bar{g}|\CCE )=h_i \IE(\bar{g}|\CCE), \mu-a.e.$ for all $\bar{g}\in L^\infty(\Omega,\AA,\mu).$ As $G_E$ is Lyapunov (Lemma \ref{JLY}, $f=1$), there is some $\bar{g}\in L^\infty(\Omega,\AA,\mu)$, $\|\bar{g}\|_\infty>0$, vanishing outside of $E$  and satisfying 
\begin{equation}\label{eq-Case1}
\IE(h_i \bar{g}|\CCE )=0,  \mu-a.e., \quad \forall i=1,\dots,n+1,.
\end{equation}
In the contrary case, there exists an $\AA$-measurable subset $A\subset E$, with $\mu(A)>0$, where all the functions $h_i^0:=h_i-\IE(h_i|\CCE)$ keep a constant sign and  where $|h^0_{i_0}|>0$, for  $\mu$-almost every $\omega\in A$, for at least one $i_0$. Note that as defined, the function $h_i^0$, $i=1,\dots,n+1,$ satisfy
\begin{equation}\label{eq-IEH0}
\IE(h_i^0|\CCE)=\IE(h_i|\CCE) -\IE(\IE(h_i|\CCE)|\CCE)=0, \mu-a.e.
\end{equation}
Without loss of generality, assume that $i_0\leq n$. Applying the induction hypothesis $P(n)$ with 
$H^0=(h_1^0,\dots,h_n^0)$ and the sub-$\sigma$-algebra $\CCE$, it follows that $H_0G_E$ is Lyapunov. By Theorem \ref{COR7}, there is $A_0\in \AA $ subset of $A$ such that
\begin{equation}\label{eq-A0HalfA}
H^0G_E(A_0)=\frac{1}{2}H^0G_E(A).
\end{equation}
Since $A$ and $A_0$ are subsets of $E$, it follows that the function $\bar{g}=  \varphi -\IE(\varphi|\CCE)$, with $\varphi=\chi_{A_0}-\chi_{A\setminus A_0},$ vanishes outside of $E$. Let us moreover show that $\|\bar{g}\|_\infty>0$. For the sake of contradiction, assume that $\bar{g}=0$ $\mu$-a.e.. Then, the $\CCE$-measurable  set $B=\{\omega\in \Omega: \IE(\varphi|\CCE)(\omega)=1\}$ is such that $\mu(B\bigtriangleup A_0)=0$. Hence, by the non-nullity of $h^0_{i_0}$ over $A_0,$ $\int_{B} h_i^0 d\mu= \int_{A_0} h_i^0d\mu\neq 0,$ which contradicts  \eqref{eq-IEH0}. Moreover, we have for $i=1,\dots,n$, 
$$
\begin{array}{rl}
\IE(h_i^0 \bar{g}|\CCE)=&\IE(h_i^0 \chi_{A_0}|\CCE)-\IE(h_i^0\chi_{A\setminus A_0}|\CCE)\\
&- \IE \left(h_i^0\IE(\varphi|\CCE)|\CCE\right)\\
=& h_i^0G_E(A_0) -h_i^0 G_E(A\setminus A_0)-\IE(\varphi|\CCE)\IE (h_i^0|\CCE)\\
=&0.
\end{array}
$$
where the last equality follows from  (\ref{eq-A0HalfA}) and (\ref{eq-IEH0}).\\
Furthermore, as $\IE(\bar{g}|\CCE)= \IE(\varphi- \IE(\varphi|\CCE)|\CCE)=0$, we have for $i=1\dots,n$,
$$\begin{array}{rl}
\IE(h_i \bar{g}|\CCE)&= \IE(h_i^0 \bar{g}|\CCE) + \IE(\IE(h_i|\CCE)\bar{g}|\CCE)\\
&=0+ \IE(h_i|\CCE) \IE(\bar{g}|\CCE)\\
&= 0.
\end{array}
$$
This still also true for $i=n+1$ by remarking 
that $h_{n+1}= \chi_E- \sum_{i=1}^n h_i$  and $\IE(\chi_E \bar{g}|\CCE) =\chi_E\IE( \bar{g}|\CCE) =0$.
Finally, by setting $g=\frac{\bar{g}}{\sum_{i=1}^{n+1}|f_i|}$ it comes that
$\IE(|f_{i}| g |\CCE)=0$, for every $i=1,...,n+1$. Then, by (\ref{eq-signFi}), $\IE(f_{i}g |\CCE)=0$, for all $i=1,\dots,n+1$.  The function $g$ has the  desired properties.
\end{proof}

\begin{proof} [Proof of Theorem \ref{THLG}]

Set $\mu=\frac{1}{n}\overset{n}{\underset{i=1}{\sum}}\mu_i$. Then, $\mu_i<\!\!<\mu$ for every $i$ and obviously $\CC$ is $\mu$-coarser than $\AA$. Observe now that $F_\CC<\!\!<\mu$. Set, for every $i$, $h_i=\frac{d\mu_i}{d\mu}$ the Radon-Nikodym derivative of $\mu_i$ relatively to $\mu$. Hence, $f_ih_i\in L^1(\Omega,\AA,\mu)$ and from Theorem \ref{theorem-induction} and Theorem \ref{TH4}, for every fixed $E\in \AA^+$, there $g\in L^\infty(\Omega,\AA,\mu)$ vanishing outside of $E$, $\|g\|_\infty>0$ and $\IE(gf_ih_i\chi_E|\CC)=0$. But, for every $A\in \CC$, $\int_A gf_ih_i\chi_E d\mu=\int_A gf_i\chi_E d\mu_i$. Therefore, $\IE^{\mu_i}(gf_i\chi_E|\CC)=0$, for every $i$. It results from Theorem \ref{TH4}, that $F_\CC$ is Lyapunov.
\end{proof}

\section*{Acknowledgment} The authors are grateful to Ali Khan for a discussion on this paper, for all his valuable comments and suggestions.

\end{document}